\documentclass[a4paper,12pt,draft]{article}
\usepackage{amsmath,amssymb,amsthm}
\newtheorem{theorem}{Theorem}
\newtheorem{lemma}{Lemma}
\newtheorem{corollary}{Corollary}
\newtheorem{proposition}{Proposition}
\newtheorem{claim}{Claim}
\theoremstyle{definition}
\newtheorem{definition}{Definition}
\newtheorem*{question}{Question}
\theoremstyle{remark}
\newtheorem*{remark}{Remark}

\newcommand{\norm}[1]{\lVert#1\rVert}
\newcommand{\abs}[1]{\lvert#1\rvert}
\newcommand{\eps}{\varepsilon}
\newcommand{\Bplus}{{\mathcal{B}^+}}
\newcommand{\aeto}{\xrightarrow{\mathrm{a.\:e.}}}
\newcommand{\xto}{\xrightarrow}
\newcommand{\xnto}{\xnrightarrow}
\newcommand{\ONE}{{\mathbf{1}}}

\newcommand\GG{{\mathbf{G}}}
\newcommand{\Edge}[1]{\mathcal{E}(#1)}
\newcommand{\Vertex}[1]{\mathcal{V}(#1)}
\newcommand{\Paths}[1]{\mathcal{P}(#1)}
\newcommand\Pop{\mathsf{P}}
\newcommand\Moper{M}
\newcommand\Mnum{{M^\circ}}
\newcommand\Tgr{T}

\makeatletter 
\def\ext@notarrow#1#2#3#4#5#6#7{%
  \mathrel{\mathop{%
    \setbox\z@\hbox{#5\displaystyle}%
    \setbox\tw@\vbox{\m@th
      \hbox{$\scriptstyle\mkern#3mu{#6}\mkern#4mu$}%
      \hbox{$\scriptstyle\mkern#3mu{#7}\mkern#4mu$}%
      \copy\z@
    }%
    \hbox to\wd\tw@{\rlap{\hbox to\wd\tw@{\hfil\raisebox{0.3ex}{$\mathbf{\scriptscriptstyle/}$}\hfil}}\unhbox\z@}}%
  \limits
    \@ifnotempty{#7}{^{\if0#1\else\mkern#1mu\fi
                       #7\if0#2\else\mkern#2mu\fi}}%
    \@ifnotempty{#6}{_{\if0#1\else\mkern#1mu\fi
                       #6\if0#2\else\mkern#2mu\fi}}}%
}
\newcommand{\xnrightarrow}[2][]{\ext@notarrow 0359\rightarrowfill@{#1}{#2}}
\newcommand{\xnleftarrow}[2][]{\ext@notarrow 3095\leftarrowfill@{#1}{#2}}
\makeatother

\begin{document}
\title{Ces\`aro convergence of spherical averages\\
for measure-preserving actions\\
of Markov semigroups and groups}
\author{Alexander Bufetov\thanks{Steklov Mathematical Institute, Moscow, Russia, and
Rice University, Houston, Texas,~USA.},
Mikhail Khristoforov\thanks{Chebyshev Laboratory, Department of Mathematics and Mechanics,
Saint-Petersburg State University, Saint-Petersburg, Russia.},
Alexey Klimenko\thanks{Steklov Mathematical Institute, Moscow, Russia.}}
\date{}
\maketitle

\begin{abstract}
Ces\`aro convergence of spherical averages is proven for
measure-preserving actions of Markov semigroups and groups.
Convergence in the mean is established for functions in $L^p$, $1\le p<\infty$,
and pointwise convergence for functions in $L^\infty$. In
particular, for measure-preserving actions of word hyperbolic groups
(in the sense of Gromov)
we obtain Ces\`aro convergence of spherical averages with respect to
any symmetric set of generators.
\end{abstract}

\section{Introduction}

\subsection{Formulation of the main results}

Let $\Gamma$ be a finitely generated semigroup. Choice of a finite set
of generators~$O$ endows~$\Gamma$ with a norm $\abs{\,\cdot\,}_O$:
for $g\in\Gamma$ the number $\abs{g}_O$ is the length of~the~shortest word
over the alphabet~$O$ representing~$g$. Denote $S_O(n)=\{g:\nobreak\abs{g}_O=n\}$.

Assume that the semigroup~$\Gamma$ acts on a probability space~$(X,\nu)$
by~measure-preserving transformations, and for $g\in\Gamma$ let $T_g$ be the
corresponding map. Now take $\varphi\in L^1(X,\nu)$ and consider the sequence of
its \emph{spherical averages}
\begin{equation*}
s_n(\varphi)=\frac1{\#S_O(n)}\sum_{g\in S_O(n)} \varphi\circ T_g
\end{equation*}
(here and everywhere $\#$ stands for the cardinality of a finite set;
if $S_O(n)=\varnothing$, then we set $s_n(\varphi)=0$).
Next, consider the Ces\`aro averages of the spherical averages:
\begin{equation*}
c_N(\varphi)=\frac1N\sum_{n=0}^{N-1}s_n(\varphi).
\end{equation*}
The main result of this paper establishes mean convergence of the
averages~$c_N(\varphi)$ for $\varphi\in L^1(X,\nu)$ and pointwise convergence
of $c_N(\varphi)$
for~$\varphi\in\nobreak L^\infty(X,\nu)$ in the case when $\Gamma$ is
a \emph{Markov semigroup} with respect to the~generating set~$O$.

Recall the definition of Markov semigroups. As before, let~$\Gamma$ be
a semigroup with a finite generating set~$O$. For a finite directed graph~$\GG$
with the~set of arcs~$\Edge{\GG}$, a labelling on~$\GG$ is a map
$\xi\colon\Edge{\GG}\to O$. Let $v_0$ be a vertex of~$\GG$ and let $\Paths{\GG,v_0}$ be
the set of all finite paths in~$\GG$ starting at~$v_0$. To~each path
$p=e_1\dots e_n\in\Paths{\GG,v_0}$ we assign an element~$\xi(p)\in\Gamma$ by
the~formula
\begin{equation*}
\xi(p)=\xi(e_1)\dots\xi(e_n).
\end{equation*}
The semigroup~$\Gamma$ is called \emph{Markov} with respect to a finite
generating set~$O$ if there exists a finite directed graph~$\GG$, a vertex~$v_0$
of~$\GG$, and a labelling $\xi\colon\Edge{\GG}\to O$ such that the lifted map
$\xi\colon\Paths{\GG,v_0}\to\Gamma$ is a bijection, and, furthermore,
for a path $p\in\Paths{\GG,v_0}$ of length~$n$ we have $\abs{\xi(p)}_O=n$.

For example, a theorem by Gromov \cite{Gromov} states that a word
hyperbolic group is Markov with respect to any symmetric set of generators
(for cocompact groups of isometries of Lobachevsky spaces, the Markov property
had been established earlier by Cannon \cite{Cannon}; a detailed exposition
of the proof of Gromov's theorem can be found in the book of Ghys and
de~la~Harpe~\cite{GhysDelaharpe}).\looseness=1

We are now ready to formulate the main result of the paper.

\begin{theorem}\label{Thm:MarkSemi}Let~$\Gamma$ be a Markov semigroup
with respect
to a finite generating
set~$O$. Assume that $\Gamma$ acts by measure-preserving transformations
on a probability space~$(X,\nu)$. Then for any~$p$, $1\le p<\infty$, and any $\varphi\in L^p(X,\nu)$
the sequence of Ces\`aro averages of its spherical averages
\begin{equation*}
c_N(\varphi)=\frac1N\sum_{n=0}^{N-1}\frac1{\#S_O(n)}\sum_{g\in S_O(n)}\varphi\circ T_g
\end{equation*}
converges in~$L^p(X,\nu)$ as~$N\to\nobreak\infty$. If, additionally,
$\varphi\in L^\infty(X,\nu)$, then the~sequence~$c_N(\varphi)$
converges $\nu$-almost everywhere as $N\to\infty$.\end{theorem}

\begin{corollary}\label{Cor:GromovGr}Let~$\Gamma$ be an infinite word
hyperbolic group \textup(in the sense of~Gromov\textup), and let $O$ be a finite symmetric generating set for~$\Gamma$.
Assume that $\Gamma$ acts by measure-preserving transformations
on a probability space~$(X,\nu)$. Then for any~$p$, $1\le\nobreak p<\infty$, and any $\varphi\in L^p(X,\nu)$
the sequence of Ces\`aro averages of its spherical averages
\begin{equation*}
c_N(\varphi)=\frac1N\sum_{n=0}^{N-1}\frac1{\#S_O(n)}\sum_{g\in S_O(n)}\varphi\circ T_g
\end{equation*}
converges in $L^p(X,\nu)$ as $N\to\infty$. If, additionally,
$\varphi\in L^\infty(X,\nu)$, then the~sequence $c_N(\varphi)$
converges $\nu$-almost everywhere as $N\to\infty$.
\end{corollary}

Under additional assumption of exponential mixing of the action,
pointwise Ces\`aro convergence for spherical averages of functions from~$L^2$
for~measure-preserving actions of word hyperbolic groups
was obtained by Fujiwara and Nevo \cite{FuNe}.
L.~Bowen \cite{Bowen} proved convergence of spherical averages for actions
of word hyperbolic groups on finite spaces. Both Fujiwara and Nevo \cite{FuNe}
and L.~Bowen~\cite{Bowen} also proved that in their setting the limit
is invariant under the action.

Our result applies to all measure-preserving actions
of all finitely-generated infinite word hyperbolic groups. Our argument,
however, does not give any information about the limit.

\begin{question}In Theorem~\ref{Thm:MarkSemi},
when is it true that the limit is $\Gamma$-invariant?\end{question}

We conjecture that it always is in Corollary~\ref{Cor:GromovGr}.

\subsection{History}

First ergodic theorems for measure-preserving actions of arbitrary
countable groups were obtained by Oseledets in 1965 \cite{Oseled}. Oseledets
endows a countable group~$\Gamma$ with a probability distribution $\mu$
satisfying $\mu(g)=\mu(g^{-1})$, $g\in\Gamma$, and establishes pointwise
convergence of the sequence of operators
\begin{equation*}
S^{(\mu)}_{2n}=\sum_{g\in\Gamma}\mu^{(2n)}(g)T_g
\end{equation*}
as $n\to\infty$ (here $\mu^{(k)}$ stands for the $k$-th convolution of
the measure~$\mu$). To~prove  pointwise convergence Oseledets
uses the martingale theorem in~the~space of trajectories of the Markov
chain corresponding to the self-adjoint Markov operator $S_1^{(\mu)}$;
the argument of Oseledets is thus a precursor, in~the~self-adjoint case,
of Rota's ``Alternierende Verfahren'' argument \cite{Rota}.

For uniform spherical averages corresponding to measure-preserving actions
of free groups convergence in the mean was established by Y. Guivarc'h \cite{Guivarch},
who used earlier work of Arnold and Krylov \cite{ArnKr} on equidistribution
of~two rotations of the sphere.

In 1986, R.I. Grigorchuk \cite{Grig86} (see also \cite{Grig99}, \cite{Grig00})
obtained pointwise convergence of Ces\`aro averages of uniform spherical
averages of $L^1$-functions for~measure-preserving actions of free groups.
The limit is invariant under the~action of the group.

For functions in $L^2$, pointwise convergence of uniform spherical averages
themselves was established in 1994 by Nevo \cite{Nevo94}, and for functions
in $L^p$, $p>1$, by Nevo and Stein \cite{NeSt}. The limit was proven to be
invariant under the subgroup of elements of even length.
Whether convergence of uniform
spherical averages holds for functions in $L^1$ remains an open problem
(recall that, as Ornstein showed \cite{Ornst}, powers of a self-adjoint
Markov operator applied to a function in $L^1$ need not converge almost surely).

In \cite{Buf02}, pointwise convergence of uniform spherical averages
is obtained by applying Rota's ``Alternierende Verfahren'' Theorem to a
special Markov operator assigned to the action. This approach also yields
pointwise convergence of \emph{non-uniform} spherical averages corresponding
to Markovian weights satisfying a symmetry condition \cite{Buf02}.

Convergence of Ces\`aro averages on non-uniform spherical averages
for~actions of free groups and free semigroups holds for general Markovian
(and, in fact, for general stationary) weights \cite{Buf99}, \cite{Buf00},
\cite{Buf01}. The motivation behind considering such Markovian weights is
precisely to establish
ergodic theorems for actions of Markov groups, in particular, of word
hyperbolic groups.

The results of \cite{Buf01}, however, can only be applied to groups
that are coded by admissible words in an \emph{irreducible} Markov chain;
in fact, to prove invariance of the limit function, even a stronger condition
is needed, which is called \emph{strict irreducibility} in \cite{Buf01}
and is equivalent to the triviality of the symmetric $\sigma$-algebra of the
corresponding Markov chain with finitely many states.

For some groups, a Markov coding is known explicitly: for instance, for
Fuchsian groups such a coding has been constructed by Series~\cite{Series}.
The Series coding does in fact have the strict irreducibility property, and
pointwise convergence of Ces\`aro averages of uniform spherical averages
for~measure-preserving actions of Fuchsian groups and for functions
in~$L^1$ is established in \cite{BufSer}, extending the earlier theorem of
Fujiwara and Nevo \cite{FuNe} for functions in~$L^2$.\looseness=-1

For general word hyperbolic groups, however, it is not clear whether
the~Markov coding is irreducible.
The main result of this paper is that convergence of Ces\`aro averages of
spherical averages still holds without the irreducibility assumption.

\subsection{Acknowledgements}

We are deeply grateful to R.I.~Grigorchuk, V.A.~Kaimanovich, A.~Nevo,
and A.M.~Vershik for useful discussions.

A.\,B. is an Alfred P. Sloan Research Fellow.
He is supported in part by~grant MK-4893.2010.1 of
the President of the Russian Federation,
by the Programme on Mathematical Control Theory of the Presidium of the Russian Academy of Sciences,
by the Programme 2.1.1/5328 of the Russian Ministry of Education and Research,
by the Edgar Odell Lovett Fund at Rice University,
by the NSF under grant DMS 0604386,
and by the RFBR-CNRS grant~\mbox{10-01-93115}.

M.\,K. is supported in part by the Chebyshev Laboratory (Department of~Mathematics
and Mechanics, Saint-Petersburg State University) under the grant
11.G34.31.2006 of the Government of the Russian Federation.

A.\,K. is partially supported by RFBR grant 08-01-00342a,
by grants \mbox{NSh-8508.2010.1} and MK-4893.2010.1 of the President
of the Russian Federation, and by the Programme on Mathematical Control Theory
of the Presidium of the Russian Academy of Sciences.

\section{Paths and operators}

Let us introduce some notation regarding a directed graph from the
definition of Markov groups.
Consider a finite directed graph $\GG$ (loops and multiple edges are permitted).
The sets of vertices and edges (arcs) of~$\GG$ are denoted as~$\Vertex{\GG}$
and $\Edge{\GG}$ respectively. For an edge $e$, $I(e)$ and $F(e)$ are its
initial (tail) and terminal (head)
vertices. Denote
\begin{equation*}
\Edge{\GG,u,v}=\{e\in\Edge{\GG}\mid I(e)=u, F(e)=v\}.
\end{equation*}
Then, let $\Paths{\GG}$ be the set of finite paths in $\GG$, that is,
\begin{equation*}
\Paths{\GG}=\{l=e_1e_2\ldots e_k\mid I(e_j)=F(e_{j-1})\}.
\end{equation*}
Denote by $\abs{l}$ the length of a path~$l$.

Let $(X,\nu)$ be a probability space. Assume that to every arc~$e\in\Edge{\GG}$
a~measure-preserving transformation $\Tgr_e$ of $(X,\nu)$ is assigned.
In this case we~say that $\GG$ is \emph{labelled} by measure-preserving transformations
of $(X,\nu)$.

The map $e\mapsto\Tgr_e$ is naturally extended onto $\Paths{\GG}$ by
formula
\begin{equation*}
\Tgr_{e_1\dots e_k}=\Tgr_{e_1}\dots\Tgr_{e_k}.
\end{equation*}
The action of $\Tgr_l$, $l\in\Paths{\GG}$,
induces a standard action on the space $L^p(X,\nu)$: $\Tgr_l(\varphi)=\varphi\circ\Tgr_l$.
For any finite subset $L\subset\Paths{\GG}$ introduce an operator $s(L)$ on
$L^p(X,\nu)$ acting by the formula
\begin{equation*}
s(L)=\frac1{\#L}\sum_{l\in L} \Tgr_l
\end{equation*}
if $L\ne\varnothing$; we set $s(\varnothing)=0$.

In particular, denote
\begin{equation*}
L^\GG_{u,v,n}=\{l\in\Paths{\GG}\mid I(l)=u, F(l)=v, |l|=n\},
\end{equation*}
that is, $L^\GG_{u,v,n}$ is the set of all paths from~$u$ to~$v$ of length~$n$.
Define $s^\GG_{u,v,n}=s(L^\GG_{u,v,n})$ and
let $c^\GG_{u,v,N}$ be their Ces\`aro averages:
\begin{equation*}
c^\GG_{u,v,N}=\frac1N\sum_{n=0}^{N-1} s^\GG_{u,v,n}.
\end{equation*}
Analogously, denote $L^\GG_{u,{*},n}=\bigcup_{v\in\Vertex{\GG}} L^\GG_{u,v,n}$
and define
\begin{equation*}
s^\GG_{u,{*},n}=s(L^\GG_{u,{*},v}),\quad c^\GG_{u,{*},N}=\frac1N\sum_{n=0}^{N-1}s^\GG_{u,{*},n}.
\end{equation*}

\begin{theorem}\label{Thm:UVU*Avg}Let $\GG$ be a finite directed graph labelled
by measure-preserving transformations of a probability space $(X,\nu)$.
Then for operators $c^\GG_{u,v,N}$ and $c^\GG_{v_0,*,N}$ defined above,
the following statements hold.\\
\textup{1a.} For any $\varphi\in L^p(X,\nu)$, $p\in [1,\infty)$,
the sequence $\{c^\GG_{u,v,N}(\varphi)\}_{N=1}^\infty$
converges in $L^p(X,\nu)$.\\
\textup{1b.} For any $\varphi\in L^\infty(X,\nu)$ the sequence
$\{c^\GG_{u,v,N}(\varphi)\}_{N=1}^\infty$ converges $\nu$-almost everywhere.\\
\textup{2a.} For any $\varphi\in L^p(X,\nu)$, $p\in [1,\infty)$,
the sequence $\{c^\GG_{v_0,{*},N}(\varphi)\}_{N=1}^\infty$
converges in $L^p(X,\nu)$.\\
\textup{2b.} For any $\varphi\in L^\infty(X,\nu)$ the sequence
$\{c^\GG_{v_0,{*},N}(\varphi)\}_{N=1}^\infty$ converges $\nu$-almost everywhere.
\end{theorem}

Statements 2a--b of  Theorem \ref{Thm:UVU*Avg} immediately imply Theorem~\ref{Thm:MarkSemi}.
Indeed, if we assign the map $T_{\xi(e)}$ to an edge $e$,
then
\begin{equation*}
s_n(\varphi)=s^\GG_{v_0,{*},n}(\varphi).
\end{equation*}

Now we proceed to the proof of Theorem~\ref{Thm:UVU*Avg}.
Define a square matrix $\Moper(\GG)$
of order $\#\Vertex{\GG}$ with entries being operators on $L^1(X,\nu)$
by the formula
\begin{equation*}
\Moper(\GG)_{u,v}=\sum_{e\in\Edge{\GG,u,v}} \Tgr_e.
\end{equation*}
Denote also $\Mnum(\GG)_{u,v}=\# \Edge{\GG,u,v}$.
Note that if $\ONE$ is the function that
equals $1$ everywhere, then $\Tgr_e\ONE=\ONE$ for any $e\in\Edge{\GG}$. Define
the following class of operators.

\begin{definition}\label{Def:B+}\emph{A class $\Bplus$} of operators
on $L^1(X,\nu)$ is a set of all operators $A\colon L^1(X,\nu)\to L^1(X,\nu)$ such that
\begin{enumerate}
\item there exists $\lambda(A)\in\mathbb R$ such that $A(\ONE)=\lambda(A)\cdot\ONE$,
\item if $f\ge 0$ (that is, $f(x)\ge 0$ for almost all $x\in X$) then $Af\ge 0$,
\item $A(L^p(X,\nu))\subset L^p(X,\nu)$ for all $p\in[1,\infty]$,
\item $\norm{Af}_p\le\lambda(A)\norm{f}_p$ for any $p\in [1,\infty]$,
$f\in L^p(X,\nu)$.
\end{enumerate}\end{definition}

It is clear that this class is a convex cone, that is, it is closed
under linear combinations with nonnegative coefficients. Since all
$\Tgr_e$'s
belong to this class, the same is true for $\Moper(\GG)_{u,v}$, and
\begin{equation*}
\lambda(\Moper(\GG)_{u,v})=\sum_{e\in\Edge{\GG,u,v}} \lambda(\Tgr_e)=
\sum_{e\in\Edge{\GG,u,v}} 1=\Mnum(\GG)_{u,v}.
\end{equation*}
Then, consider an $n$-th power of the graph~$\GG$, that is,
a graph $\GG'=\GG^n$, where $\Vertex{\GG'}=\Vertex{\GG}$,
$\Edge{\GG'}=\{l\in\Paths{\GG}, \abs{l}=n\}$, and $I(l)=I(e_1)$, $F(l)=F(e_n)$
for $l=e_1\dots e_n\in\Edge{\GG'}$.

By definition, $\Moper(\GG^n)_{u,v}=\sum_{l\in L^\GG_{u,v,n}} \Tgr_l$.
It is also clear that $(\Moper(\GG))^n=\Moper(\GG^n)$, and
\begin{equation*}
\lambda((\Moper(\GG)^n)_{u,v})=(\Mnum(\GG)^n)_{u,v}=\#L^\GG_{u,v,n}.
\end{equation*}
Now if we define an operation $\Pop$ on the class~$\Bplus$
as $\Pop(T)=T/\lambda(T)$ if $T\ne 0$, $\Pop(0)=0$, then we have
\begin{equation*}
s^\GG_{u,v,n}=\Pop((\Moper(\GG^n))_{u,v}),\quad
c^\GG_{u,v,N}=\frac1N\sum_{n=0}^{N-1} \Pop((\Moper(\GG^n))_{u,v}).
\end{equation*}
Similarly,
\begin{equation*}
s^\GG_{v_0,{*},n}=\Pop\Biggl(\sum_{v\in\Vertex{\GG}}(\Moper(\GG^n))_{v_0,v}\Biggr),\quad
c^\GG_{v_0,{*},N}=\frac1N\sum_{n=0}^{N-1} \Pop\Biggl(\sum_{v\in\Vertex{\GG}}(\Moper(\GG^n))_{v_0,v}\Biggr).
\end{equation*}

\section{The Main Lemma}

The proof of statements 1a--b of Theorem~\ref{Thm:UVU*Avg} is obtained through
a decomposition
of the graph~$\GG$ into smaller blocks. The basic (non-decomposable) situation
is the case of a~strongly connected graph (that is, a graph such that for
any its vertices~$u,v$ there exists a path from~$u$ to~$v$)
and in this case the theorem is proven in~\cite{Buf01}. A step of
the procedure starts with a decomposition of the~set~$\Vertex{\GG}$ into two disjoint
nonempty sets~$V_1$, $V_2$ with no arcs from $V_2$ to $V_1$.
Then we apply Theorem~\ref{Thm:UVU*Avg} to the induced subgraphs with
these sets of vertices (that is, a graphs $\GG_i$, $i=1,2$,
with $\Vertex{\GG_i}=V_i$ and
$\Edge{\GG_i}=\{e\in\nobreak\Edge{\GG_i}:I(e),F(e)\in V_i\}$), and use
Lemma~\ref{Lem:OperPreConv} (see below), which is the~main technical
statement of the paper. The statements 2a--b of Theorem~\ref{Thm:UVU*Avg}
are deduced from the statements 1a--b using the same lemma.

\begin{definition}\label{Def:regular}A sequence $\{x_n\}_{n=0}^\infty$,
$x_n\ge 0$, is called \emph{regular} if there exists a number $q\in\mathbb N$ such that for
each $r=0,\dots,q-1$
one of the following statements holds:
\begingroup
\makeatletter\renewcommand\theenumi{(\@alph\c@enumi)}%
\renewcommand\labelenumi{\theenumi}\makeatother
\begin{enumerate}
\item $x_{qk+r}=0$ for all but finite number of $k\ge 0$,
\item $\lim\limits_{k\to\infty} \dfrac{x_{qk+r}}{a k^b c^k}=1$
for some $a>0$, $b\in\mathbb N$, $c\ge 1$.
\end{enumerate}\endgroup
\end{definition}

\begin{definition}\label{Def:PreConv}
A sequence $\{T_n\}_n$, $T_n\in\Bplus$, is called \emph{pre-convergent} if
\begin{enumerate}
\item the sequence $\{\lambda(T_n)\}_n$ is regular;
\item for any $\varphi\in L^p(X,\nu)$ the sequence
$\displaystyle\biggl\{\frac1N\sum_{n=0}^{N-1}\Pop(T_n)(\varphi)\biggr\}$
converges in~$L^p(X,\nu)$ as $N\to\infty$;
\item for any $\varphi\in L^\infty(X,\nu)$ the sequence
$\displaystyle\biggl\{\frac1N\sum_{n=0}^{N-1}\Pop(T_n)(\varphi)\biggr\}$ converges
\mbox{almost} everywhere as $N\to\infty$.
\end{enumerate}
\end{definition}

In these terms, Theorem~\ref{Thm:UVU*Avg} can be reformulated as follows.

\begin{proposition}\label{Prop:MPreConv}Under conditions of
Theorem~\ref{Thm:UVU*Avg} the following
statements hold.\\
\textup{1.} For any induced subgraph~$\GG'$ of the graph~$\GG$ the
sequence $\{(\Moper(\GG')^n)_{u,v}\}_n$ is pre-convergent for any
$u,v\in\Vertex{\GG'}$.\\
\textup{2.} The sequence
\begin{equation*}
\biggl\{\sum_{v\in\Vertex{\GG}}(\Moper(\GG)^n)_{v_0,v}\biggr\}_n
\end{equation*}
is pre-convergent for any $v_0\in\Vertex{\GG}$.\end{proposition}

The first statement of Proposition \ref{Prop:MPreConv} is equivalent to the statements 1a--b of
Theorem~\ref{Thm:UVU*Avg} for all
induced subgraphs of~$\GG$. This is convenient for our inductive
argument.
The basis for the induction is the following theorem.

\begin{theorem}[\cite{Buf01}]\label{Thm:Buf}
If a graph $\GG$ is strongly connected,
then the sequence $\{(\Moper(\GG)^n)_{u,v}\}_n$ is pre-convergent for any $u,v\in\Vertex{\GG}$.
\end{theorem}

\begin{remark}1. Regularity of the sequence $\{\lambda((\Moper(\GG)^n)_{u,v})\}_n=
\{(\Mnum(\GG)^n)_{u,v}\}_n$ in the case of strongly connected graph
follows from the Perron---Frobenius theorem.\\
2. Convergence of $c^\GG_{u,v,N}$ in~$L^1(X,\nu)$ and almost everywhere
(for functions in~$L^1(X,\nu)$) is shown
in \cite{Buf01} (see Theorems 1, 2;
note that strong connectivity of~$\GG$ is called
irreducibility of~$A=\Mnum(\GG)$ in \cite{Buf01}).
$L^p$-convergence for functions in $L^p(X,\nu)$ follows
immediately.\end{remark}

The step of the inductive procedure relies on the following lemma.

\begin{lemma}\label{Lem:OperPreConv}If sequences $\{F_n\}$ and $\{G_n\}$
of operators from the class~$\Bplus$ are pre-convergent,
then the following ones are also pre-convergent:
\begingroup\renewcommand\labelenumi{\textup{\theenumi.}}
\begin{enumerate}
\item
$\{H^{(1)}_n\}_n$, $H^{(1)}_n=F_n$ for $n\ge n_0$,
$H^{(1)}_n\in\Bplus$\textup;
\item\label{Item:Shift}
$\{H^{(2)}_n=F_{n+M}\}_n$ for any $M\in\mathbb Z$\textup;
\item\label{Item:MultConst}
$\{H^{(\mathrm{3a})}_n=AF_n\}_n$, $\{H^{(\mathrm{3b})}_n=F_nA\}_n$,
where $A\in\Bplus$\textup;
\item\label{Item:Sum} $\{H^{(4)}_n=F_n+G_n\}_n$\textup;
\item\label{Item:Convolution} $\{H^{(5)}_n=\sum_{k+m=n}F_kG_m\}_n$.
\end{enumerate}
\end{lemma}
\endgroup
We now derive Proposition~\ref{Prop:MPreConv} from Lemma~\ref{Lem:OperPreConv}.

\begin{proof}[Proof of Proposition~\ref{Prop:MPreConv}.]
1. The proof of the first statement is by induction on the number of
vertices in~$\GG'$.

(a) Any graph~$\GG'$ with $\#\Vertex{\GG'}=1$ is strongly connected,
thus we can apply Theorem~\ref{Thm:Buf}.

(b) Take any induced subgraph $\GG'$ with $k$ vertices
and suppose that the~statement holds for any induced subgraph of~$\GG$
with less than $k$ vertices.
Then there are two cases:
(1)~$\GG'$~is strongly connected;
(2)~$\GG'$ can be decomposed as follows: $\Vertex{\GG'}=V_1\sqcup V_2$,
$V_{1,2}\ne\varnothing$, and there are no arcs from~$V_2$ to $V_1$.

In the first case we may apply Theorem~\ref{Thm:Buf}.
In the second case consider graphs $\GG_{1,2}$ that are induced subgraphs
with $\Vertex{\GG_i}=V_i$. Since $\GG_{1,2}$ have less that $k$ vertices,
the theorem holds for them.

Now consider $c^{\GG'}_{u,v,N}$.
If $u,v\in V_i$, $i=1,2$, a path from $u$ to $v$ can't leave~$\GG_i$, so
$(\Moper(\GG')^n)_{u,v}=(\Moper(\GG_i)^n)_{u,v}$, hence
$c^{\GG'}_{u,v,N}=c^{\GG_i}_{u,v,N}$, and the~statement is reduced
to the one for $\GG_i$. The case $u\in V_2$, $v\in V_1$ is even simpler:
there are no paths from $u$ to $v$, so $c^{\GG'}_{u,v,N}=0$ for all $N$.

The only nontrivial case is $u\in V_1$, $v\in V_2$.
Here
\begin{equation*}
(\Moper(\GG')^n)_{u,v}=\sum_{\substack{k+m=n-1\\v'\in V_1, u'\in V_2}}
(\Moper(\GG_1)^k)_{u,u'}\Moper(\GG')_{u',v'}
(\Moper(\GG_2)^m)_{v',v}
\end{equation*}
and the statement follows from Lemma~\ref{Lem:OperPreConv}. Indeed,
by assumption, the sequence $\{(\Moper(\GG_1)^n)_{u,u'}\}_n$ is pre-convergent,
hence, by item \ref{Item:MultConst}b of this lemma,
the sequence $\{(\Moper(\GG_1)^n)_{u,u'}\Moper(\GG')_{u',v'}\}_n$ is.
Then, as $\{(\Moper(\GG_2)^n)_{v',v}\}_n$ is pre-convergent by assumption,
item~\ref{Item:Convolution} gives us that
\begin{equation*}
\bigl\{A^{u,u',v,v'}_n=\sum_{k+m=n}
(\Moper(\GG_1)^k)_{u,u'}\Moper(\GG')_{u',v'}
(\Moper(\GG_2)^m)_{v',v}\bigr\}_n
\end{equation*}
is also pre-convergent. Now $\{A^{u,u',v,v'}_{n-1}\}_n$
is pre-convergent by item~\ref{Item:Shift}, and, finally the sequence
\begin{equation*}
\Bigl\{(\Moper(\GG')^n)_{u,v}=\sum_{v'\in V_1, u'\in V_2}A^{u,u',v,v'}_{n-1}\Bigr\}_n
\end{equation*}
is pre-convergent by item~\ref{Item:Sum} of Lemma~\ref{Lem:OperPreConv}.

2. The second statement follows immediately from item~\ref{Item:Sum} of
Lemma~\ref{Lem:OperPreConv}.\end{proof}

\section{Proof of Lemma~\ref{Lem:OperPreConv}}

The rest of the paper is devoted to the proof of Lemma~\ref{Lem:OperPreConv}.
The proof will often use the following proposition.

\begin{proposition}\label{Prop:B+ae-conv}
Let $A\in\Bplus$, $\varphi_n\in L^\infty(X,\nu)$, $\norm{\varphi_n}_\infty\le C$,
$\varphi_n(x)\to\varphi(x)$ for almost all $x\in X$.
Then $(A\varphi_n)(x)\to (A\varphi)(x)$
for almost all $x\in X$.
\end{proposition}

\begin{proof}Clearly, it is sufficient to prove this only for $\varphi=0$.

Further, decompose $\varphi_n$ as $\varphi_n=\varphi^+_n-\varphi^-_n$,
where $\varphi^\pm_n=\max(0,\pm\varphi_n)$, $\norm{\varphi^\pm_n}_\infty\le\norm{\varphi_n}_\infty\le C$.
Therefore, if we prove that $A\varphi^\pm_n\aeto 0$, then
$A\varphi_n=A\varphi^+_n-A\varphi^-_n\aeto 0$.
So we can assume that $\varphi_n\ge 0$.

Now, take $\psi_n(x)=\sup\{\varphi_k(x)\mid k\ge n\}$. Then $\psi_n(x)$
is monotonically nonincreasing and tends to zero for almost all~$x\in X$.
Since $0\le \varphi_n\le \psi_n$, the same is true for their images:
$0\le A\varphi_n\le A\psi_n$ and therefore, it is sufficient to prove that
$A\psi_n\aeto 0$. But as $A\psi_n(x)$ is nonnegative and nonincreasing,
there is a limit $\theta(x)=\lim_{n\to\infty} A\psi_n(x)$, and, by monotone convergence theorem,
\begin{equation*}
\norm{A\psi_n-\theta}_1=\int_X A\psi_n(x)-\theta(x)\, d\nu(x)\to 0
\end{equation*}
Therefore, $A\psi_n\to\theta$ in $L^1(X,\nu)$.
But $A$ is a bounded operator in $L^1(X,\nu)$
and $\norm{\psi_n}_1\to 0$  (also due to monotone convergence theorem),
so $A\psi_n\to 0$ in~$L^1(X,\nu)$.
Thus $\theta(x)=0$ almost everywhere.\end{proof}

\begin{proof}[Proof of Lemma~\ref{Lem:OperPreConv}.]The plan of the proof
is the following. After some preparations, we'll prove the first condition
in Definition~\ref{Def:PreConv} for all sequences $\{H^{(*)}_n\}_n$ (here
and below the asterisk ${*}$ denotes one of the symbols $1$, $2$, $3\mathrm{a}$,
$3\mathrm{b}$, $4$, $5$), and then we'll prove the second and the third
condtions of that Definition simultaneously.

1. First of all, it is
sufficient to prove that this lemma
holds for the sequences $\{\lambda(F_n)\}$ and $\{\lambda(G_n)\}$ satisfying
Definition~\ref{Def:regular} with $q=1$ (and that in this case
the sequence $\{\lambda(H^{(*)}_n)\}$ is also regular with $q=1$).

Indeed, in general case we take $q$ to be the least common multiple
of~$q_F$ and $q_G$ (i.~e., $q$'s from Definition~\ref{Def:regular} for
the sequences $\{F_n\}_n$ and $\{G_n\}_n$). For~${*}\ne\nobreak 5$, it is clear that
for a given $r=0,\dots, q-1$ the sequence
$\{H^{(*)}_{qs+r}\}_s$
depends in the same fashion on one of $\{F_{qs+r'}\}_s$
and $\{G_{qs+r''}\}_s$ with some~$r',r''$.

Now consider ${*}=5$. Let $k=qu+r'$, $m=qv+r''$ ($u,v\ge 0$, $0\le  r',r''\le q-1$) and decompose the
sum
\begin{equation*}
\sum_{k+m=qs+r}F_kG_m
\end{equation*}
into $q$~sums
corresponding to all possible pairs $(r',r'')$ (there are only $q$ possibilities,
since $r'+r''\equiv r\pmod{q}$):
\begin{equation*}
H^{(5)}_{qs+r}=\sum_{r'+r''\equiv r\:(\mathrm{mod}\:q)}S^{r',r''}_s,
\end{equation*}
where
\begin{equation*}
S^{r',r''}_s=\sum_{\substack{u,v\ge 0\\(qu+r')+(qv+r'')=qs+r}}F_{qu+r'}G_{qv+r''}=
\sum_{\substack{u,v\ge 0\\u+v=s+\frac{r-r'-r''}{q}}}F_{qu+r'}G_{qv+r''},
\end{equation*}
that is, the sequence $\{S^{r',r''}_s\}_s$ is the convolution of the
sequences $\{F_{qs+r'}\}_s$
and $\{G_{qs+r''}\}_s$ shifted by $\frac{r-r'-r''}q\in\{-1,0\}$.

\medskip

2. Let us prove that the sequences $\{\lambda(H^{(*)}_n)\}_n$ are regular.
For ${*}=1,2,\mathrm{3a},\mathrm{3b}$ this is clear from the definitions.
Let ${*}=4$.
If $\{\lambda(F_n)\}$
or $\{\lambda(G_n)\}$ contains only finitely many nonzero elements, this
is clear.
Otherwise, let $(a_F,b_F,c_F)$ and $(a_G,b_G,c_G)$ be the constants
given in Definition~\ref{Def:regular}
for these sequences.

If (1)~$c_G<c_F$ or (2)~$c_G=c_F$, $b_G<b_F$, then
\begin{equation*}
\lim_{k\to\infty} \frac{\lambda(G_k)}{a_F k^{b_F}c_F^k}=0,
\end{equation*}
so
\begin{equation*}
\lim_{k\to\infty} \frac{\lambda(F_k+G_k)}{a_F k^{b_F}c_F^k}=1.
\end{equation*}
The symmetric cases ($1'$)~$c_F<c_G$;
($2'$)~$c_F=c_G$, $b_F<b_G$ are similar.
The only remaining case is $c_F=c_G=c$, $b_F=b_G=b$. Here
\begin{equation*}
\lim_{k\to\infty} \frac{\lambda(F_k+G_k)}{(a_F+a_G) k^b c^k}=1.
\end{equation*}

Now let ${*}=5$. The case of finitely many nonzeros is again clear,
otherwise we can assume that $c_F\ge c_G$. There are two cases, $c_F>c_G$
and $c_F=c_G$.

Suppose that $c_F>c_G$.
Then
\begin{equation*}
\frac{\sum\limits_{k+m=n}\lambda(F_k)\lambda(G_m)}{a_F (n+1)^{b_F}c_F^n}=
\sum_{m=0}^n \lambda(G_m)\biggl(\frac{\lambda(F_{n-m})}{a_F (n+1)^{b_F}c_F^n}\biggr).
\end{equation*}
Let us prove that this sum tends to $\sum_{m=0}^\infty \lambda(G_m)c_F^{-m}$.

Denote\footnote{We write $(n+1)^b$ in the denominator
instead of $n^b$ to have well-defined $\alpha_0$. Nevertheless,
$\alpha_n$ tends to $1$.}
\begin{equation*}
\alpha_n=\frac{\lambda(F_n)}{a_F (n+1)^{b_F}c_F^n},\quad
\beta_n=\frac{\lambda(G_m)}{c_F^m}
\end{equation*}
and fix $\eps>0$.
Note that the series $\sum_{m=0}^\infty \lambda(G_m)c_F^{-m}=
\sum_{m=0}^\infty\beta_n$ converges absolutely, so there is $m_0$ such that
$\sum_{m>m_0} \beta_m<\eps$.
Let $A$ be an upper bound for all $\alpha_n$, $n\ge 1$ (it exists
since $\alpha_n\to1$).
Then
\begin{multline*}
\Biggl|\sum_{m=0}^n \beta_m\alpha_{n-m}\Bigl(\frac{n-m+1}{n+1}\Bigr)^{b_F}-
\sum_{m=0}^\infty \beta_m\Biggr|\le{}\\
\begin{aligned}
{}\le{}&\Biggl|\sum_{m=0}^{m_0} \beta_m\biggl[\alpha_{n-m}\Bigl(1-\frac{m}{n+1}\Bigr)^{b_F}-1\biggr]\Biggr|+\\
&{}+\Biggl|\sum_{m=m_0+1}^{n}\beta_m\alpha_{n-m}\Bigl(1-\frac{m}{n+1}\Bigr)^{b_F}\Biggr|+
\Biggl|\sum_{m=m_0+1}^{\infty}\beta_m\Biggr|.
\end{aligned}
\end{multline*}
The last term is less than $\eps$, the second one is less than $A\eps$ and,
if $n$ is sufficiently large, the first term is less than $\eps$,
hence the whole difference is less than $(2+A)\eps$ for sufficiently large $n$.
Thus, $\{\lambda(H^{(5)}_n\}$ is regular with
\begin{equation*}
a_H=a_F\sum_{m=0}^\infty\lambda(G_m)c_F^{-m},\quad b_H=b_F,\quad c_H=c_F.
\end{equation*}

Now let $c_F=c_G=c$. In this case we have
\begin{multline}\label{eq:lambda5b}
\frac{\sum_{k+m=n}\lambda(F_k)\lambda(G_m)}{a_Fa_G (n+1)^{b_F+b_G+1}c^n}={}\\
\frac1{n+1}\sum_{k=0}^{n}
\underbrace{\vphantom{\frac\strut{\strut c^{n-k}_G}}\frac{\lambda(F_k)}{\strut a_F(k+1)^{b_F}c_F^k}}_{\alpha_k}
\underbrace{\vphantom{\frac\strut{\strut c^{n-k}_G}}\frac{\lambda(G_{n-k})}{\strut a_G(n-k+1)^{b_G}c_G^{n-k}}}_{\beta_{n-k}}
\underbrace{\vphantom{\frac\strut{\strut c^{n-k}_G}}\biggl(\frac{k+1}{\strut n+1}\biggr)^{b_F}\Bigl(1-\frac k{\strut n+1}\Bigr)^{b_G}}_{\gamma_{n,k}}
\end{multline}
and denote $\alpha_k$, $\beta_{n-k}$ and $\gamma_{n,k}$ as it is shown here.
Let us show that
\begin{equation*}
\lim_{n\to\infty}\frac1{n+1}\sum_{k=0}^n \alpha_k\beta_{n-k}\gamma_{n,k}-
\frac1{n+1}\sum_{k=0}^n \gamma_{n,k}=0.
\end{equation*}
Indeed, by Definition~\ref{Def:regular}, the
sequences $\{\alpha_k\}$ $\{\beta_k\}$ tends to~$1$,
hence there are $A,B$ such that $\alpha_k\le A$, $\beta_k<B$ for all $k$.
Take any $\eps<1$ and find $p$ such that $\abs{\alpha_k-1}<\eps$,
$\abs{\beta_k-1}<\eps$
for all $k\ge p$.
Then
\begin{multline*}
\Delta_n=\frac1{n+1}\sum_{k=0}^n \alpha_k\beta_{n-k}\gamma_{n,k}-
\frac1{n+1}\sum_{k=0}^n \gamma_{n,k}={}\\
\frac1{n+1}\Biggl(
\sum_{k=0}^{p-1}+\sum_{k=p}^{n-p} +\sum_{k=n-p+1}^{n}
\Biggr)
(\alpha_k\beta_{n-k}-1)\gamma_{n,k}.
\end{multline*}
Since $0\le\gamma_{n,k}\le 1$, any term of the first and the last sums
is bounded by $AB+1$ and any term of the middle sum is bounded by $2\eps+\eps^2\le 3\eps$.
Therefore, we have
\begin{equation*}
\Delta_n\le \frac{2p(AB+1)+3\eps(n+1-2p)}{n+1}\le 3\eps+\frac{2(AB+1)p}{n+1}.
\end{equation*}
If $n$ is large enough then the last term is less than $\eps$,
hence~$\Delta_n\le 4\eps$.

It remains to find the limit
\begin{equation*}
\lim_{n\to\infty}\frac1{n+1}\sum_{k=0}^n\gamma_{n,k}.
\end{equation*}
We have
\begin{equation*}
\frac1{n+1}\sum_{k=0}^n\gamma_{n,k}=\Bigl(\frac{n+2}{n+1}\Bigr)^{b_F+b_G+1}
\biggl(\frac 1{n+2}
\sum_{j=1}^{n+1}\Bigl(\frac j{n+2}\Bigr)^{b_F}
\Bigl(1-\frac j{n+2}\Bigr)^{b_G}\biggr).
\end{equation*}
The first multiplier tends to~$1$. The second one equals the
Riemann sum of~the~function $f(x)=x^{b_F}(1-x)^{b_G}$ over the unit interval
with the partition
\begin{equation*}
\Bigl\{x_i=\frac{i}{n+2}\Bigr\}_{i=0}^{n+2},\quad \{t_i=x_i\}_{i=0}^{n+1},
\end{equation*}
hence it tends to $\int_0^1 f(x)\,dx=B(b_F+1,b_G+1)$.
Thus, in this case $\{\lambda(H^{(5)}_n\}$ is regular with the constants
\begin{equation*}
a_H=a_Fa_G B(b_F+1,b_G+1),\quad b_H=b_F+b_G+1,\quad c_H=c.
\end{equation*}

\medskip

3. We proceed to the proof of the second and the third conditions
in~Definition~\ref{Def:PreConv}.

For ${*}=1,2$ the difference between Ces\`aro sums satisfies the relations
\begin{align*}
\frac1N\sum_{n=0}^{N-1}\Pop(H^{(1)}_n)-\frac1N\sum_{n=0}^{N-1}\Pop(F_n)&{}=
\frac1N\sum_{n=0}^{n_0-1}\bigl(\Pop(H^{(1)}_n)-\Pop(F_n)\bigr),\\
\frac1N\sum_{n=0}^{N-1}\Pop(H^{(2)}_n)-\frac1N\sum_{n=0}^{N-1}\Pop(F_n)&{}=
\frac1N\sum_{n=0}^{M-1}\bigl(\Pop(F_{N+n})-\Pop(F_n)\bigr),
\end{align*}
whence it tends to zero even in operator norm in any $L^p(X,\nu)$, $p\in[1,\infty]$.

For ${*}=\mathrm{3a},\mathrm{3b}$ the conditions follows from the identities
\begin{align*}
\frac1N\sum_{n=0}^{N-1}\Pop(H^{(\mathrm{3a})}_n)&{}=\Pop(A)\biggl(\frac1N\sum_{n=0}^{N-1}\Pop(F_n)\biggr),\\
\frac1N\sum_{n=0}^{N-1}\Pop(H^{(\mathrm{3b})}_n)&{}=\biggl(\frac1N\sum_{n=0}^{N-1}\Pop(F_n)\biggr)\Pop(A).
\end{align*}

The only remaining cases are ${*}=4,5$. Let us show that we can make
``approximate'' normalisations of operators instead of ``precise''
ones (that~is,~$\Pop(\,\cdot\,)$) in the second and the third conditions
in Definition~\ref{Def:PreConv}. Speaking formally, the following holds.

\pagebreak

\begin{claim}\label{Clm:ApprNorm}Suppose that the sequence $\{T_n\}_n$,
$T_n\in\Bplus$, satisfies the condition
\begin{equation*}
\lim_{n\to\infty}\frac{\lambda(T_n)}{an^bc^n}=1
\end{equation*}
with some $a>0$, $b\in\mathbb N$, and $c\ge1$. Let
\begin{equation*}
\hat T_n=\frac{T_n}{a(n+1)^bc^n}.
\end{equation*}
Then for any $\varphi\in L^p(X,\nu)$ the sequences
\begin{equation*}
C_N(\varphi)=\frac1N\sum_{n=0}^{N-1}\Pop(T_n)(\varphi)\quad\text{and}\quad
C'_N(\varphi)=\frac1N\sum_{n=0}^{N-1}\hat T_n(\varphi)
\end{equation*}
converge \textup(in $L^p$ or a.~e.\textup) simultaneously and their limits
coincide.
\end{claim}

\begin{proof}If $\lambda(\hat T_n)=\gamma_n$, $\gamma_n\to 1$, then
we have
\begin{multline*}
\norm{C_N-C_N'}_p=\Biggl\|\frac1N\sum_{n=0}^{N-1}(\Pop(T_n)-\hat T_n)\Biggr\|_p=
\Biggl\|\frac1N\sum_{n=0}^{N-1}\Pop(T_n)(1-\gamma_n)\Biggr\|_p\le{}\\
{}\le\frac1N\sum_{n=0}^{N-1}\|\Pop(T_n)\|_p\cdot |1-\gamma_n|\le
\frac1N\sum_{n=0}^{N-1}|1-\gamma_n|,
\end{multline*}
and the latter is the Ces\`aro sum of $x_n=|1-\gamma_n|$, which tends to zero.
Thus the difference $C_N-C_N'$ tends to zero in operator norm in any
$L^p(X,\nu)$, $p\in[1,\infty]$.
\end{proof}

Now let ${*}=4$. If one of the sequences $\{F_n\}$, $\{G_n\}$ has only
finitely many nonzero terms, we can use the lemma's statement
for $\{H^{(1)}_n\}_n$.
Otherwise take the constants $a_F,b_F,c_F$, $a_G,b_G,c_G$ same as before and
introduce operators $\hat F_n$, $\hat G_n$, $\hat H^{(4)}_n$ in the same way as
in Claim~\ref{Clm:ApprNorm}.

From the previous section of the proof
one can see that $\hat H^{(4)}_n$ is
either $\frac{a_F}{a_F+a_G}\hat F_n+\frac{a_G}{a_F+a_G}\hat G_n$
(if $c_F=c_G$ and $b_F=b_G$), or $\hat F_n+\eps_n\hat G_n$ with $\eps_n\to 0$
(if $c_G<c_F$, or if $c_G=c_F$ and $b_G<b_F$), or $\eps_n\hat F_n+\hat G_n$
(in symmetric cases). The convergence of Ces\`aro sums of~$\hat H^{(4)}_n$
in the first case is obvious, in the two latter cases
the term $\eps_n\hat F_n$ (or $\eps_n\hat G_n$) tends to zero in operator norm:
\begin{equation*}
\norm{\eps_n\hat F_n}\le \eps_n (\norm{F_n}/a_Fn^{b_F}c_F^n)\to 0\cdot 1,
\end{equation*}
and so does the sequence of its Ces\`aro averages.

Finally, suppose ${*}=5$. As usual, the proof is clear
if $\{F_n\}$ or $\{G_n\}$ contains finitely many nonzero terms,
otherwise let $a_{F,G,H}$, $b_{F,G,H}$, $c_{F,G,H}$ be the coefficients
in the regularity condition respectively for $\{\lambda(F_n)\}$, $\{\lambda(G_n)\}$,
$\{\lambda(H^{(5)}_n)\}$. Similarly to the case ${*}=4$, we'll prove
convergence for~the~sequence
\begin{equation*}
\bigl\{\frac1N\sum_{n=0}^{N-1}\hat H^{(5)}_n(\varphi)\bigr\}_N.
\end{equation*}

There are three cases, $c_F>c_G$, $c_F<c_G$, and $c_F=c_G$.
Suppose the first one. Then $c_H=c_F$, $b_H=b_F$, and for
any $\varphi\in L^p(X,\nu)$ we have
\begin{multline}\label{eq:H5case1}
C_N(\varphi)=\frac1N\sum_{n=0}^{N-1}\hat H^{(5)}_n(\varphi)=
\frac1N\sum_{k+m<N}\frac{a_F}{a_H}\hat F_k
\Bigl(\frac{G_m(\varphi)}{c_F^m}\Bigr)\cdot\Bigl(\frac {k+1}{k+m+1}\Bigr)^{b_F}={}\\
\sum_{m=0}^{N-1}\underbrace{\frac1N\sum_{k<N-m}
\frac{a_F}{a_H}\hat F_k
\Bigl(\frac{G_m(\varphi)}{c_F^m}\Bigr)\cdot\Bigl(\frac {k+1}{k+m+1}\Bigr)^{b_F}}_{S_{m,N}(\varphi)}.
\end{multline}
Let $A$ be chosen in such a way that $\|\hat F_k\|\le A$ for all $k$.
Then we have
\begin{multline*}
\norm{S_{m,N}(\varphi)}_p\le
\frac1N\sum_{k<N-m}\frac{a_F}{a_H}\norm{\hat F_k}_p
\frac{\lambda(G_m)}{c_F^m}\norm{\varphi}_m \Bigl(\frac {k+1}{k+m+1}\Bigr)^{b_F}\le{}\\
\frac1N\sum_{k<N-m}\frac{a_F}{a_H}\cdot A\cdot \frac{\lambda(G_m)}{c_F^m}
\cdot \norm{\varphi}_m\cdot 1\le
\frac{Aa_F\norm{\varphi}_p}{a_H}\cdot \frac{\lambda(G_m)}{c_F^m}.
\end{multline*}
Since $\sum_{m=0}^\infty \lambda(G_m)c_F^{-m}<\infty$, we can choose $m_0$
such that
\begin{equation}\label{eq:H5case1m0}
\sum_{m>m_0}\lambda(G_m)c_F^{-m}<\eps\cdot\frac{a_H}{a_F\norm{\varphi}_p}.
\end{equation}
Then we have
\begin{equation}\label{eq:H5case1estA}
\biggl\|\sum_{m>m_0}S_{m,N}(\varphi)\biggr\|_p\le A\eps.
\end{equation}
Further, let us find the limit of $S_{m,N}(\varphi)$ as $N\to\infty$.
Denote
\begin{equation*}
\psi_m=\frac{a_F}{a_H}\frac{G_m(\varphi)}{c_F^m},
\end{equation*}
then
\begin{multline}\label{eq:SmN}
S_{m,N}(\varphi)=\frac{N-m}{N}\biggl(
\frac1{N-m}\sum_{k<N-m}\hat F_k(\psi_m)-{}\\
-\frac1{N-m}\sum_{k<N-m}\hat F_k(\psi_m)\Bigl[1-\Bigl(\frac {k+1}{k+m+1}\Bigr)^{b_F}\Bigr]
\biggr).
\end{multline}
Due to regularity of the sequence $\{F_k\}$, the first term in parentheses
tends in $L^p$ or a.~e.\ to a function, which will be denoted as $F^0(\psi_m)$.
Note also that the equality \begin{equation*}
F^0(\theta)=\lim_{n\to\infty}\frac1N\sum_{k=0}^{N-1} \Pop(F_k)(\theta)=
\lim_{n\to\infty}\frac1N\sum_{k=0}^{N-1} \hat F_k(\theta)
\end{equation*}
defines a linear operator~$F^0\in\Bplus$, with $\lambda(F^0)=1$.

The second term in parentheses in \eqref{eq:SmN} is the Ces\`aro average
for the sequence
\begin{equation*}
\theta_{m,k}=\hat F_k(\psi_m)\Bigl[1-\nobreak\Bigl(\frac {k+1}{k+m+1}\Bigr)^{b_F}\Bigr],
\end{equation*}
which tends to zero in $L^p(X,\nu)$, $p\in[1,\infty]$, as $k\to\infty$.
Thus $S_{m,N}\to\nobreak F^0(\psi_m)$  in $L^p$ or a.~e.
In particular, there exists $N_m$ such that for any $N>N_m$ we~have
\begin{equation}\label{eq:H5case1estB}
\norm{S_{m,N}(\varphi)-F^0(\psi_m)}_p\le \frac{\eps}{m_0+1}.
\end{equation}
Similarly, for $\varphi\in L^\infty(X,\nu)$, for almost all~$x$ there exists $N_m(x)$
such that for~any~$N>N_m(x)$ we have
\begin{equation}\tag{\ref{eq:H5case1estB}${}'$}\label{eq:H5case1estB'}
\abs{S_{m,N}(\varphi)(x)-F^0(\psi_m)(x)}\le \frac{\eps}{m_0+1}.
\end{equation}
Note also that \eqref{eq:H5case1m0} yields
\begin{equation*}
\biggl\|\sum_{m>m_0}\psi_m\biggr\|_p\le
\sum_{m>m_0}\frac{a_F}{a_H}\norm{\varphi}_p\frac{\lambda(G_m)}{c_F^m}<\eps,
\end{equation*}
whence, noting that $\norm{F^0}_p\le1$, we obtain
\begin{equation}\label{eq:H5case1estC}
\Biggl\|\sum_{m=0}^{m_0} F^0(\psi_m)-F^0\biggl(\frac{a_F}{a_H}\sum_{m=0}^\infty
\frac{G_m(\varphi)}{c_F^m}\biggr)\Biggr\|_p<\eps.
\end{equation}
Now we can see that
\begin{align*}
C_N(\varphi)-F^0\biggl(\frac{a_F}{a_H}\sum_{m=0}^\infty
\frac{G_m(\varphi)}{c_F^m}\biggr)={}
&\sum_{m>m_0}S_{m,N}(\varphi)+{}\\
&\sum_{m=0}^{m_0}(S_{m,N}(\varphi)-F^0(\psi_m))+{}\\
&\biggl(\sum_{m=0}^{m_0} F^0(\psi_m)-F^0\biggl(\frac{a_F}{a_H}\sum_{m=0}^\infty
\frac{G_m(\varphi)}{c_F^m}\biggr)\biggr),
\end{align*}
and, if $N>\max(N_0,\dots,N_{m_0})$, the estimates \eqref{eq:H5case1estA},
\eqref{eq:H5case1estB}, \eqref{eq:H5case1estC} give us the~inequality
\begin{equation*}
\biggl\|C_N(\varphi)-F^0\biggl(\frac{a_F}{a_H}\sum_{m=0}^\infty
\frac{G_m(\varphi)}{c_F^m}\biggr)\biggr\|<(2+A)\eps.
\end{equation*}
Similarly, if $N>\max(N_0(x),\dots,N_{m_0}(x))$ for $\varphi\in L^\infty(X,\nu)$
then \eqref{eq:H5case1estA},
\eqref{eq:H5case1estB'}, and \eqref{eq:H5case1estC} imply
\begin{equation*}
\biggl|C_N(\varphi)(x)-F^0\biggl(\frac{a_F}{a_H}\sum_{m=0}^\infty
\frac{G_m(\varphi)}{c_F^m}\biggr)(x)\biggr|<(2+A)\eps.
\end{equation*}

The second case $c_F<c_G$ is treated similarly. Namely,
the sum for $C_N(\varphi)$ is decomposed into the sums
\begin{equation}\label{eq:H5case2SkN}
S_{k,N}(\varphi)=\frac{a_G}{a_H}\frac1{c_G^k}F_k\biggl(
\frac1N\sum_{m<N-k}\hat G_m(\varphi)\Bigl(\frac{m+1}{k+m+1}\Bigr)^{b_G}
\biggr).
\end{equation}
The estimate of its norm for $k>k_0$ is the same, and the only
difference is in~the~proof of convergence of $S_{k,N}(\varphi)$ as $N\to\infty$:
the argument of $F_k$ in \eqref{eq:H5case2SkN} tends to $G^0(\varphi)$,
so $S_{k,N}$ tends (in $L^p$ or a.~e.) to
\begin{equation*}
\frac{a_G}{a_H}\frac{F_k(G^0(\varphi))}{c_G^k},
\end{equation*}
with Proposition~\ref{Prop:B+ae-conv} being used in case $\varphi\in L^\infty(X,\nu)$.

Now consider the third case $c_F=c_G=c$. Here
\begin{equation*}
\frac1N\sum_{n=0}^{N-1}\hat H^{(5)}_n=
\frac1N\sum_{k+m<N}\frac{1}{a_H(k+m)^{b_F+b_G+1}}\frac{F_k}{c^k}\frac{G_m}{c^m}
\end{equation*}
and the lemma follows from Proposition~\ref{Prop:Conv} for $X_n=F_nc^{-n}$,
$Y_n=G_n c^{-n}$.
\end{proof}

\begin{proposition}\label{Prop:Conv}Let $X_n,Y_n\in\Bplus$ be such that
\begingroup\renewcommand\labelenumi{\textup{\theenumi.}}
\begin{enumerate}
\item the sequences $\{\lambda(X_n)/(n+1)^u\}$ and $\{\lambda(Y_n)/(n+1)^v\}$
are bounded,
\item for any $\varphi\in L^p(X,\nu)$, $p\in[1,\infty)$,
the sequences
$\biggl\{\displaystyle\frac1N\sum\limits_{n=0}^{N-1}\frac{X_n(\varphi)}{(n+1)^u}\biggr\}$
and $\biggl\{\displaystyle\frac1N\sum\limits_{n=0}^{N-1}\frac{Y_n(\varphi)}{(n+1)^v}\biggr\}$
converge in~$L^p(X,\nu)$ as $N\to\infty$,
\vspace{-5pt}
\item for any $\varphi\in L^\infty(X,\nu)$
the sequences
$\biggl\{\displaystyle\frac1N\sum\limits_{n=0}^{N-1}\frac{X_n(\varphi)}{(n+1)^u}\biggr\}$
and\quad\strut\ $\biggl\{\displaystyle\frac1N\sum\limits_{n=0}^{N-1}\frac{Y_n(\varphi)}{(n+1)^v}\biggr\}$
converge almost everywhere as $N\to\infty$.
\end{enumerate}
Let $Z_n=\sum_{k+m=n}X_kY_m$, $w=u+v+1$. Then
\begin{enumerate}
\item for~any~$\varphi\in L^p(X,\nu)$, $p\in[1,\infty)$,
the sequence
$\biggl\{\displaystyle\frac1N\sum_{n=0}^{N-1}\frac{Z_n(\varphi)}{(n+1)^w}\biggr\}$
converges in $L^p(X,\nu)$,
\vspace{-5pt}
\item for any $\varphi\in L^\infty(X,\nu)$
the sequence $\biggl\{\displaystyle\frac1N\sum_{n=0}^{N-1}\frac{Z_n(\varphi)}{(n+1)^w}\biggr\}$
converges \mbox{almost} everywhere.
\end{enumerate}
\endgroup
\end{proposition}

\begin{proof}1. Let
\begin{equation}\label{eq:defX0}
X^0(\varphi)=\lim_{N\to\infty}\frac1N\sum_{n=0}^{N-1}\frac{X_n(\varphi)}{(n+1)^u},\quad
Y^0(\varphi)=\lim_{N\to\infty}\frac1N\sum_{n=0}^{N-1}\frac{Y_n(\varphi)}{(n+1)^v}.
\end{equation}
These operators belong to $\Bplus$. Indeed, the first two conditions
are obvious, and, to check the remaining two, one
can see that
\begin{gather*}
\norm{X^0(\varphi)}_p\le
\limsup_{N\to\infty} \frac1N\sum_{n=0}^{N-1}\frac{\norm{\lambda(X_n)\varphi}_p}{(n+1)^u}\le
\limsup_{N\to\infty} \frac1N\sum_{n=0}^{N-1}\frac{\lambda(X_n)}{(n+1)^u}\norm{\varphi}_p,\\
\abs{X^0(\varphi)(x)}\le
\limsup_{N\to\infty} \frac1N\sum_{n=0}^{N-1}\frac{\abs{X_n\varphi(x)}}{(n+1)^u)}\le
\limsup_{N\to\infty} \frac1N\sum_{n=0}^{N-1}\frac{\lambda(X_n)}{(n+1)^u}\norm{\varphi}_\infty,
\end{gather*}
and note that the sequence
\begin{equation*}
\biggl\{\frac1N\sum_{n=0}^{N-1}\frac{\lambda(X_n)}{(n+1)^u}\biggr\}_N
\end{equation*}
is bounded by the same bound as the sequence $\{\lambda(X_n)/(n+1)^u\}_n$.

\pagebreak[2]
\medskip

2. Now introduce
\begin{equation*}
X_n^*=X_n-(n+1)^uX^0,\qquad
Y_n^*=Y_n-(n+1)^vY^0.
\end{equation*}
These operators are bounded in any $L^p(X,\nu)$, $p\in[1,\infty]$, and
the norms
$\norm{X_n^*}_p/(n+1)^u$, $\norm{Y_n^*}_p/(n+1)^v$ are bounded uniformly
on $p\in[1,\infty]$ and~$n$ (indeed, these bounds are simply
twice the bounds for
$\norm{X_n}_p/(n+1)^u=\lambda(X_n)/(n+1)^u$,
$\norm{Y_n}_p/(n+1)^v=\lambda(Y_n)/(n+1)^v$).
This is an analogue of the~first condition of the proposition;
one can see that the second and
the~third conditions hold for $X_n^*$, $Y_n^*$ in place of $X_n$, $Y_n$.

Furthermore,
\begin{multline}\label{eq:Zn}
\frac{Z_n(\varphi)}{(n+1)^w}=\sum_{k+m=n}\frac{X_kY_m(\varphi)}{(k+m+1)^w}={}\\
\begin{aligned}
{}=&\sum_{k+m=n}\frac{X^*_kY^*_m(\varphi)}{(k+m+1)^w}-
\biggl(\sum_{k+m=n}\frac{(m+1)^vX^*_k}{(k+m+1)^w}\biggr)(Y^0(\varphi))-{}\\
&{}-X^0\biggl(\sum_{k+m=n}\frac{(k+1)^uY^*_m(\varphi)}{(k+m+1)^w}\biggr)+
\sum_{k+m=n}\frac{(k+1)^u(m+1)^v}{(k+m+1)^w}X^0Y^0(\varphi).
\end{aligned}
\end{multline}
To prove Proposition~\ref{Prop:Conv}, it is sufficient
to prove ($L^p$- and a.~e.-) convergence of Ces\`aro averages
for each term in~\eqref{eq:Zn}.

\medskip

3. For the last term in \eqref{eq:Zn} the proof is simple:
\begin{equation*}
\sum_{k+m=n}\frac{(k+1)^u(m+1)^v}{(k+m+1)^w}=
\Bigl(\frac{n+2}{n+1}\Bigr)^w\cdot
\biggl(\frac1{n+2}\sum_{j=1}^{n+1}\Bigl(\frac j{n+2}\Bigr)^u
\Bigl(1-\frac j{n+2}\Bigr)^v\biggr).
\end{equation*}
Here the first multiplier tends to $1$ and the second one is the
Riemann sum of $f(x)=x^u(1-x)^v$ with a partition of $[0,1]$ into
$n+2$ equal intervals, so it tends to the Euler integral $B(u+1,v+1)$.
Therefore,
the last term tends to~$B(u+1,v+1)X^0Y^0(\varphi)$ and so do its
Ces\`aro averages.

\medskip

4. To prove convergence of the second and the third terms in \eqref{eq:Zn},
it is sufficient
to prove that Ces\`aro averages of
\begin{equation}\label{eq:Zn23}
\sum_{k+m=n}\frac{(m+1)^vX^*_k(\varphi)}{(k+m+1)^w}\quad\text{and}\quad
\sum_{k+m=n}\frac{(k+1)^uY^*_m(\varphi)}{(k+m+1)^w}
\end{equation}
converge to zero in $L^p(X,\nu)$ for any $\varphi\in L^p(X,\nu)$, $p\in [1,\infty)$,
and a.~e.\ for any $\varphi\in L^\infty(X,\nu)$.
Indeed, for the second term we denote $\psi=Y^0(\varphi)$ and
for the third one we use either boundedness of the operator $X^0$
in $L^p(X,\nu)$ or Proposition~\ref{Prop:B+ae-conv}.

The expressions in \eqref{eq:Zn23} transform to another one when we
swap $X\leftrightarrow Y$, $u\leftrightarrow v$, and $k \leftrightarrow m$,
so we may deal only with the first of them.

Denote
\begin{equation*}
A_n=\frac1n\sum_{k=0}^{n-1}\frac{X_k^*}{(k+1)^u},\quad \varphi_n=A_n(\varphi).
\end{equation*}
By construction, $\varphi_n$ tends to~$0$ in
$L^p(X,\nu)$ for $\varphi\in L^p(X,\nu)$ and
almost everywhere for $\varphi\in L^\infty(X,\nu)$.
Further,
\begin{equation*}
X_n^*=(n+1)^u\bigl((n+1)A_{n+1}-nA_n\bigr),
\end{equation*}
thus
\begin{multline*}
C_N=\frac1N\sum_{n=0}^{N-1}\sum_{k+m=n}\frac{(m+1)^vX^*_k(\varphi)}{(k+m+1)^w}={}\\
\frac1N\sum_{k+m<N}\frac{(m+1)^v(k+1)^u}{(k+m+1)^w}
\bigl((k+1)\varphi_{k+1}-k\varphi_k\bigr),
\end{multline*}
and, rearranging the sum, we have
\begin{equation}\label{eq:Zn23a}
C_N=\sum_{k=1}^{N}
\frac kN\biggl(\sum_{m=0}^{N-k}\frac{(m+1)^vk^u}{(m+k)^w}-
\sum_{m=0}^{N-k-1}\frac{(m+1)^v(k+1)^u}{(m+k+1)^w}\biggr)\varphi_k.
\end{equation}
Now we'll use the following statement.

\begin{claim}\label{Clm:alpha_xi}Let $\alpha_{N,k}\in\mathbb R$, $\xi_k\in\Xi$,
where $\Xi$ is a normed space. Suppose that
\begingroup\renewcommand\labelenumi{\textup{\theenumi.}}
\begin{enumerate}
\item $\xi_k\to 0$ as $k\to\infty$,
\item for any fixed $N$, there are only finitely many $k$'s
with $\alpha_{N,k}\ne 0$,
\item for any fixed $k$, $\alpha_{N,k}\to 0$ as $N\to\infty$,
\item there is such $C$ that\/ $\sum_k\abs{\alpha_{N,k}}<C$ for any $N$.
\end{enumerate}
\endgroup
Then\/ $\sum_k\alpha_{N,k}\xi_k\to 0$ as $N\to\infty$.
\end{claim}

\begin{proof}[Proof of Claim~\ref{Clm:alpha_xi}]
Let $\norm{\xi_k}<R$ for any $k$.
Take any $\eps>0$ and choose $k_0$ in~such~a~way that $\norm{\xi_k}<\eps$
for $k>k_0$. Since
\begin{equation*}
\sum_{k\le k_0}\abs{\alpha_{N,k}}\to 0\quad\text{as}\quad N\to\infty,
\end{equation*}
we can choose $N_0$ such that for any $N>N_0$
\begin{equation*}
\sum_{k\le k_0}\abs{\alpha_{N,k}}<\eps.\pagebreak[2]
\end{equation*}
Therefore, for any $N>N_0$ we have
\begin{equation*}
\Bigl\|\sum_k\alpha_{N,k}\xi_k\Bigr\|\le
\sum_k\abs{\alpha_{N,k}}\norm{\xi_k}=
\sum_{k\le k_0}\abs{\alpha_{N,k}}\norm{\xi_k}+
\sum_{k> k_0}\abs{\alpha_{N,k}}\norm{\xi_k}\le
\eps R + C\eps,
\end{equation*}
and the claim is established.
\end{proof}

We apply Claim~\ref{Clm:alpha_xi} to~\eqref{eq:Zn23a}
either with $\xi_k=\varphi_k$, $\Xi=L^p(X,\nu)$ (if~$\varphi\in\nobreak L^p(X,\nu)$)
or with $\xi_k=\varphi_k(x)$, $\Xi=\mathbb R$
(if $\varphi\in L^\infty(X,\nu)$). Obviously, $\xi_k\to 0$, and we need
to check conditions on $\alpha_{N,k}$, where
\begin{equation}\label{eq:alphaNk}
\alpha_{N,k}=\frac kN\biggl[
\frac{(N-k+1)^vk^u}{N^w}+{}\!\sum_{m=0}^{N-k-1}(m+1)^v\biggl(\frac{k^u}{(m+k)^w}-
\frac{(k+1)^u}{(m+k+1)^w}\biggr)\biggr]
\end{equation}
for $k=1,\dots,N$, otherwise $\alpha_{N,k}=0$.
The value in round brackets is of~the~form $f(k)-f(k+1)$ for
$f(x)=x^u/(x+m)^w$,
so we apply the mean value theorem to it.

There are two cases: $u>0$ and $u=0$. In the first case,
\begin{multline*}
\biggl|\frac{k^u}{(m+k)^w}-\frac{(k+1)^u}{(m+k+1)^w}\biggr|=
|f'(x_m)|=\frac{x_m^{u-1}|um-(v+1)x_m|}{(x_m+m)^{w+1}}\le{}\\
{}\le\frac{(k+1)^{u-1}}{(m+k)^w}\frac{um+(v+1)x_m}{x_m+m}\le
\frac{(k+1)^{u-1}}{(m+k)^w}(u+v+1)
\end{multline*}
(here $x_m\in [k,k+1]$). Thus we have
\begin{multline*}
\abs{\alpha_{N,k}}\le
\frac{k^{u+1}(N-k+1)^v}{N^{w+1}}+
\frac{kw}{N}\sum_{m=0}^{N-k-1}\frac{(m+1)^v(k+1)^{u-1}}{(m+k)^w}\le{}\\
{}\le\frac1N+\frac{k(k+1)^{u-1}w}{N}
\sum_{m=0}^{N-k-1}\frac{1}{(m+k)^{u+1}}
\end{multline*}
The sum $\sum_{j=k}^\infty j^{-(u+1)}$ is estimated as
\begin{multline*}
\sum_{j=k}^\infty \frac1{j^{u+1}}=\frac1{k^{u+1}}+
\sum_{j=k+1}^\infty \frac1{j^{u+1}}\le{}\\
\frac1{k^{u+1}}+
\int_k^{+\infty}\frac{dx}{x^{u+1}}=
\frac1{k^{u+1}}+\frac1{uk^{u}}\le\Bigl(1+\frac1u\Bigr)\frac1{k^u}.
\end{multline*}
Continue estimation for $\abs{\alpha_{N,k}}$:
\begin{multline*}
\abs{\alpha_{N,k}}\le
\frac1N\biggl(1+\frac{w(1+u)}u\cdot\frac{k(k+1)^{u-1}}{k^u}\biggr)={}\\
\frac1N\biggl(1+\frac{w(1+u)}u\cdot\Bigl(\frac{k+1}k\Bigr)^{u-1}\biggr)\le
\frac1N\biggl(1+\frac{w(1+u)}u\cdot 2^{u-1}\biggr).
\end{multline*}
Hence $\alpha_{N,k}\to 0$ as $N\to\infty$ for any fixed~$k$ , and
\begin{equation*}
\sum_k\abs{\alpha_{N,k}}\le 1+\frac{w(1+u)}u\cdot 2^{u-1}.
\end{equation*}
Thus in the case $u>0$ all conditions of~Claim~\ref{Clm:alpha_xi} hold.

Now let $u=0$. Here
\begin{equation*}
\biggl|\frac1{(m+k)^w}-\frac1{(m+k+1)^w}\biggr|=
\frac{|-w|}{(x_m+m)^{w+1}}\le
\frac{w}{(m+k)^{w+1}},
\end{equation*}
and
\begin{multline*}
\abs{\alpha_{N,k}}\le
\frac1N+\frac kN\sum_{m=0}^{N-k-1} w\frac{(m+1)^v}{(k+m)^{v+2}}\le{}\\
\frac1N+\frac{kw}N\sum_{m=0}^{N-k-1}\frac 1{(k+m)^2}\le
\frac1N+\frac{kw}N\cdot \frac2k=\frac{1+2w}N,
\end{multline*}
hence $\alpha_{N,k}\to 0$ as $N\to\infty$ and $\sum_k\abs{\alpha_{N,k}}\le 1+2w$.

\medskip

5. It remains to consider the first term in~\eqref{eq:Zn}. Denote
\begin{equation*}
A_n=\frac1n\sum_{k=0}^{n-1}\frac{X_k^*}{(k+1)^u},\quad
B_n=\frac1n\sum_{k=0}^{n-1}\frac{Y_k^*}{(k+1)^v},
\end{equation*}
hence
\begin{equation*}
X_n^*=(n+1)^u\bigl((n+1)A_{n+1}-nA_n\bigr),\quad
Y_n^*=(n+1)^u\bigl((n+1)B_{n+1}-nB_n\bigr).
\end{equation*}
Therefore, this term equals
\begin{multline*}
\tilde C_N=\frac1N\sum_{n=0}^{N-1}\sum_{k+m=n}\frac{X^*_kY^*_m(\varphi)}{(k+m+1)^w}={}\\
\shoveleft{\phantom{\tilde C_N}=\frac1N\sum_{k+m\le N-1}
\frac{(k+1)^u(m+1)^v}{(k+m+1)^w}}\times{}\\
{}\times((k+1)A_{k+1}-kA_k)((m+1)B_{m+1}-mB_m)(\varphi).
\end{multline*}
Rearranging the terms we obtain\footnote{Here we use \emph{Iverson bracket notation}:
for any statement $\mathcal A$
\begin{equation*}[\mathcal A]=\begin{cases}1,&\mathcal A\text{ is true,}\\
0,&\mathcal A\text{ is false.}\end{cases}
\end{equation*}}
\begin{multline*}
\tilde C_N=
\frac1N\sum_{k,m\ge1}
\biggl(\frac{k^um^v}{(k+m-1)^w}[k+m\le N+1]\\
{}-\frac{(k+1)^um^v}{(k+m)^w}[k+m\le N]-
\frac{k^u(m+1)^v}{(k+m)^w}[k+m\le N]\\
{}+\frac{(k+1)^u(m+1)^v}{(k+m+1)^w}[k+m\le N-1]\biggr)kmA_kB_m(\varphi).
\end{multline*}
This sum $\tilde C_N$ is decomposed as $\tilde C_N=\tilde C_N^{(1)}+\tilde C_N^{(2)}$,
where
\begin{subequations}
\begin{multline}\label{eq:tildeC1}
\tilde C_N^{(1)}=
\frac1N\sum_{\substack{k,m\ge1\\k+m\le N}}
\biggl(\frac{k^um^v}{(k+m-1)^w}-
\frac{(k+1)^um^v}{(k+m)^w}-{}\\
-\frac{k^u(m+1)^v}{(k+m)^w}+
\frac{(k+1)^u(m+1)^v}{(k+m+1)^w}\biggr)kmA_kB_m(\varphi),
\end{multline}
\begin{multline}
\tilde C_N^{(2)}=
\frac1N\biggl(
\sum_{\substack{k,m\ge1\\k+m=N+1}}\frac{k^um^v}{(k+m-1)^w}kmA_kB_m(\varphi)-\\
\sum_{\substack{k,m\ge1\\k+m=N-1}}\frac{(k+1)^u(m+1)^v}{(k+m+1)^w}kmA_kB_m(\varphi)\biggr).
\end{multline}
\end{subequations}
We'll prove that both $\tilde C^{(1)}_N$ and $\tilde C^{(2)}_N$ tend
to zero in $L^p(X,\nu)$ for $\varphi\in L^p(X,\nu)$, $p\in [1,\infty)$,
or almost everywhere for $\varphi\in L^\infty(X,\nu)$.

Let us start with $\tilde C^{(1)}_N$. Denote $g(x,y)=x^uy^v/(x+y-1)^w$,
then the~expression in round brackets in \eqref{eq:tildeC1} equals
\begin{multline*}
\bigl(g(k,m)-g(k+1,m)\bigr)-\bigl(g(k,m+1)-g(k+1,m+1)\bigr){}\\
=-g'_y(k,\mu)+g'_y(k+1,\mu)=g''_{xy}(\varkappa,\mu),
\end{multline*}
where $\varkappa\in(k,k+1)$, $\mu\in(m,m+1)$.
(We apply the mean value theorem first to $h_1(y)=g(k,y)-g(k+1,y)$
and then to $h_2(x)=g'_y(x,\mu)$.)
One can see that
\begin{equation*}
\begin{aligned}
g''_{xy}(\varkappa,\mu)={}&uv\frac{\varkappa^{u-1}\mu^{v-1}}{(\varkappa+\mu-1)^w}-
vw\frac{\varkappa^u\mu^{v-1}}{(\varkappa+\mu-1)^{w+1}}-{}\\
&-vw\frac{\varkappa^{u-1}\mu^v}{(\varkappa+\mu-1)^{w+1}}+
w(w+1)\frac{\varkappa^u\mu^v}{(\varkappa+\mu-1)^{w+2}}.
\end{aligned}
\end{equation*}
As $\varkappa>k\ge1$, $\mu>m\ge 1$, we have $\varkappa,\mu\le \varkappa+\mu-1$,
so each fraction\footnote{We cannot use this estimate when exponent
$u-1$ (resp., $v-1$) is negative, but then $u$ (resp., $v$) equals zero,
and the estimate~\eqref{eq:gxy} is simply $0\le 0$ for this term.}
is not more than $1/(\varkappa+\mu-1)^3$, thus
\begin{multline}\label{eq:gxy}
\abs{g''_{xy}(\varkappa,\mu)}\le\frac {uv+vw+uw+w(w+1)}{(\varkappa+\mu-1)^3}\\
{}\le\frac {uv+vw+uw+w(w+1)}{(k+m-1)^3}=\frac {\Theta_{u,v}}{(k+m-1)^3}.
\end{multline}

Now we proceed to an estimation of $A_kB_m(\varphi)$.

\begin{claim}\label{Clm:Mnto0}\textup{1.} Let $M_n=\sup_{k+m=n} \norm{A_kB_m(\varphi)}_p$
for some $\varphi\in L^p(X,\nu)$, $p\in\nobreak[1,\infty)$. Then $M_n\to 0$.\\
\textup{2.} Let $M_n(x)=\sup_{k+m=n} \abs{(A_kB_m(\varphi))(x)}$ for some
$\varphi\in L^\infty(X,\nu)$. Then~$M_n(x)\to\nobreak 0$ for almost all $x\in X$.
\end{claim}

\begin{proof}1. Let $\norm{A_k}_p\le C$ for all $k$.
Denote $\varphi_m=B_m(\varphi)$. Since $\norm{\varphi_m}_p\to 0$,
for a given $\eps>0$ one can choose $m_0$ such
that $\norm{\varphi_m}_p<\eps$ for all $m\ge m_0$.
Then $\norm{A_{n-m}\varphi_m}_p\le C\eps$ for $m\ge m_0$, so
\begin{equation*}
M_n\le\max(\norm{A_n\varphi_0}_p,\norm{A_{n-1}\varphi_1}_p,\dots,
\norm{A_{n-m_0}\varphi_{m_0}}_p,C\eps).
\end{equation*}
Since $\norm{A_n\varphi_m}_p\to 0$ as $n\to\infty$ for any fixed $m$,
there are $N_0,\dots,N_{m_0}$ such that $\norm{A_{n-m}(\varphi_m)}_p\le\eps$
for $n>N_m$, $m=0,\dots, m_0$. Therefore if~$n\ge\nobreak N=\nobreak\max(N_0,\dots,N_{m_0})$,
then $M_n\le\max(\eps,C\eps)$.

2. Now let $\varphi\in L^\infty(X,\nu)$. Since $\varphi_m\aeto 0$,
if we denote
\begin{equation*}
\psi_r(x)=\max_{m\ge r}\abs{\varphi_m(x)},
\end{equation*}
then $\psi_r\aeto 0$. Note that $\psi_r(x)$ is nonnegative and nonincreasing sequence
for any $x\in X$.

The operators $A_k$ need not belong to $\Bplus$. But if we denote
\begin{equation*}
A^+_k(\theta)=\frac 1N\sum_{n=0}^{n-1}\frac{X_n(\theta)}{(n+1)^u},
\end{equation*}
then $A^+_k\in\Bplus$, and $A^+_k(\theta)\aeto X^0(\theta)$
for any $\theta\in L^\infty(X,\nu)$ (by definition of $X^0$,
see \eqref{eq:defX0}). It is also clear that $A_k=A^+_k-X^0$.

Now define the following ``exceptional sets'':
\begin{align*}
E^1&{}=\{x\mid X^0(\psi_r)(x)\xnto{\!r\to\infty\!}0\},\\
E^2_m&{}=\{x\mid A_k(\varphi_m)(x)\xnto{\!k\to\infty\!}0\},\\
E^3_r&{}=\{x\mid A_k(\psi_r)(x)\xnto{\!k\to\infty\!}0\},\\
E^4_k&{}=\{x\mid A_k(\varphi_m)(x)\xnto{\!m\to\infty\!}0\},
\end{align*}
Their measure is zero due to Proposition~\ref{Prop:B+ae-conv}
(for $E^1$, $E^4_k$)
and since $A_k(\theta)\aeto 0$ (for $E^2_m$, $E^3_r$).
Denote $E=E^1\cup\Bigl(\bigcup_m E^2_m\Bigr)\cup\Bigl(\bigcup_r E^3_r\Bigr)
\cup\Bigl(\bigcup_k E^4_k\Bigr)$
and prove that $M_n(x)\to 0$ for any $x\in X\setminus E$.

Indeed, take any $\eps>0$. Choose $r_0$ such that $X^0(\psi_{r_0})(x)\le\eps$
(here we use that $x\notin E^1$).
Note that since $X^0\in\Bplus$, $X^0(\psi_r)\ge0$ for any $r$
and $X^0(\psi_r)\le X^0(\psi_{r_0})\le\eps$ for any $r\ge r_0$.

Now choose $k_0$ such that $\abs{A_k(\psi_{r_0})(x)}<\eps$ for any $k>k_0$
($x\notin E^3_{r_0}$).
Then all possible $k$'s are divided into three classes, each class is estimated separately.

\noindent\textbf{Case 1. }Let $k=0,\dots,k_0$. Then, since $x\notin E^4_k$, there exists
$N^{(1)}_k$ such that $\abs{A_k(\varphi_{n-k})(x)}<\eps$ for any $n>N^{(1)}_k$.
Choose $N^{(1)}=\max(N^{(1)}_0,\dots,N^{(1)}_{k_0})$. Then for any $n>N^{(1)}$
\begin{equation*}
M^{(1)}_n(x)=\max_{\substack{k+m=n\\k\le k_0}}\abs{A_kB_m(\varphi)(x)})\le\eps.
\end{equation*}

\noindent\textbf{Case 2. }Let $k=k_0+1,\dots,n-r_0$. Then
\begin{multline*}
\abs{A_k(\varphi_{n-k})(x)}\le
\abs{A_k^+(\varphi_{n-k})(x)}+\abs{X^0(\varphi_{n-k})(x)}\le{}\\
{}\le A_k^+(\psi_{r_0})(x)+X^0(\psi_{r_0})(x)\le
2X^0(\psi_{r_0})(x)+\abs{A_k(\psi_{r_0})(x)}\le 2\eps+\eps=3\eps.
\end{multline*}
Thus,
\begin{equation*}
M^{(2)}_n(x)=\max_{\substack{k+m=n\\k_0<k\le n-r_0}}
\abs{A_kB_m(\varphi)(x)}\le3\eps.
\end{equation*}

\noindent\textbf{Case 3. }Let $k=n-r_0+1,\dots, n$. Then,
since $A_n(\varphi_m)(x)\xto{n\to\infty} 0$
for~any~$m=\nobreak0,\dots,\allowbreak r_0-\nobreak1$  (we use that $x\notin E^2_m$),
one can choose $N^{(3)}_m$ such that
$\abs{A_{n-m}(\varphi_m)(x)}<\eps$ for any $n>N^{(3)}_m$, $m=0,\dots,r_0-1$.
Thus, for any $n>N^{(3)}=\max(N^{(3)}_0,\dots, N^{(3)}_{r_0-1})$
\begin{equation*}
M^{(3)}_n(x)=\max_{\substack{k+m=n\\k>n-r_0}}\abs{A_kB_m(\varphi)(x)})\le\eps.
\end{equation*}

Putting these estimates together, we obtain
that
\begin{equation*}
M_n(x)=\max(M^{(1)}(x),M^{(2)}(x),M^{(3)}(x))\le3\eps
\end{equation*}
for $n>N=\max(N^{(1)},N^{(3)})$.\end{proof}

Combining \eqref{eq:gxy} with Claim~\ref{Clm:Mnto0}, we have
\begin{multline*}
\bigl\|\tilde C^{(1)}_N\bigr\|_p\le \frac 1N\sum_{\substack{k,m\ge1\\k+m\le N}}
\frac{\Theta_{u,v}km}{(k+m-1)^3}M_{k+m}\\{}\le
\frac 1N\sum_{\substack{k,m\ge1\\k+m\le N}} \frac{\Theta_{u,v}}{k+m-1}M_{k+m}=
\frac {\Theta_{u,v}}N\sum_{n=2}^N\sum_{\substack{k,m\ge1\\k+m=n}}\frac{M_n}{n-1}=
\frac {\Theta_{u,v}}N\sum_{n=2}^N M_n
\end{multline*}
and $\frac 1N\sum_{n=2}^N M_n\xto{N\to\infty} 0$ as
Ces\`aro averages of the sequence $\{M_n\}$, which converges to zero.
For a.~e.-convergence this proof also works after substitution of
$\abs{\tilde C^{(1)}_N(x)}$ for $\norm{\tilde C^{(1)}_N}$
and of $M_n(x)$ for $M_n$.

Now we estimate $\tilde C^{(2)}_N$.
\begin{multline}\label{eq:C2N}
\begin{aligned}
\tilde C^{(2)}_N=\frac 1N\biggl(&
\sum_{m=1}^N \tfrac{(N+1-m)^um^{v+1}}{N^w}(N+1-m)A_{N+1-m}B_m(\varphi)\\
&{}-\sum_{m=1}^{N-1} \tfrac{(N-m)^u(m+1)^vm}{N^w}(N-1-m)A_{N-1-m}B_m(\varphi)
\biggr)={}
\end{aligned}
\\
\begin{aligned}
{}={}&\frac1N
\sum_{m=1}^{N-1}
\Bigl[\tfrac{(N+1-m)^um^{v+1}}{N^w}-\tfrac{(N-m)^u(m+1)^vm}{N^w}\Bigr]
(N-1-m)A_{N-1-m}B_m(\varphi)+{}\\
&\begin{aligned}{}+\frac1N
\sum_{m=1}^{N-1}
\dfrac{(N+1-m)^um^{v+1}}{N^w}
\bigl[
&(N+1-m)A_{N+1-m}B_m(\varphi)\\
&\quad{}-(N-1-m)A_{N-1-m}B_m(\varphi)\bigr]+{}\end{aligned}
\\
&\begin{aligned}{}+\frac1{N^{u+1}}A_NB_1(\varphi)\end{aligned}\\
\end{aligned}
\end{multline}
Convergence of the last term is immediate. For the first term we apply
Claim~\ref{Clm:Mnto0}. Indeed, the expression in square brackets
is of the form $$m(f(m)-f(m+1)),$$ and the mean value theorem yields that
(here $\mu\in[m,m+1]$)
\begin{multline*}
\biggl|\frac{(N+1-m)^um^{v+1}}{N^w}-\frac{(N-m)^u(m+1)^vm}{N^w}\biggr|={}\\
{}=\frac m{N^w}\bigl|-u(N+1-\mu)^{u-1}\mu^v+v(N+1-\mu)^u\mu^{v-1}\bigr|\le{}\\
{}\le m\biggl(\frac{u(N+1-\mu)^{u-1}\mu^v}{N^w}+\frac{v(N+1-\mu)^u\mu^{v-1}}{N^w}\biggr)\le
m\frac{u+v}{N^2}\le\frac{u+v}N,
\end{multline*}
whence $L^p$-norm of the first term is bounded by
\begin{multline*}
\frac1N\sum_{m=1}^{N-1}\frac{u+v}{N}(N-1-m)\norm{A_{N-1-m}B_m(\varphi)}\le{}\\
{}\le\frac1N\sum_{m=1}^{N-1}(u+v)M_{N-1}\le (u+v)M_{N-1}
\end{multline*}
so it tends to zero. The same argument works for a.e.-convergence,
with $L^p$\nobreakdash-norm being replaced by absolute value of value at $x$
and $M_{N-1}$ being replaced by $M_{N-1}(x)$.

As for the second term in~\eqref{eq:C2N},
the coefficient $(N+1-m)^um^{v+1}/N^w$ is bounded by~$1$,
and the expression in square brackets equals
\begin{multline}\label{eq:ABdiff}
(N+1-m)A_{N+1-m}B_m(\varphi)-(N-1-m)A_{N-1-m}B_m(\varphi)={}\\
{}=\biggl(\frac{X^*_{N-m}}{(N+1-m)^u}+\frac{X^*_{N+1-m}}{(N+2-m)^u}\biggr)(B_m(\varphi)).
\end{multline}
Denote
\begin{equation*}
W_k=\frac{X_k}{(k+1)^u}+\frac{X_{k+1}}{(k+2)^u}
\end{equation*}
Then the sequence
\begin{equation*}
\frac 1N\sum_{k=0}^{N-1}W_k(\varphi)
\end{equation*}
tends to $2X^0(\varphi)=W^0(\varphi)$
in $L^p(X,\nu)$ (for $\varphi\in L^p(X,\nu)$) or a.~e.\ (for~$\varphi\in\nobreak L^\infty(X,\nu)$),
hence \eqref{eq:ABdiff} is equal to $(W_{N-m}-\nobreak W^0)(B_m(\varphi))$.

\begin{claim}\label{Clm:Wk}\textup{1.} If $\varphi\in L^p(X,\nu)$, then
\begin{equation*}
S_N=\frac1N\sum_{m=1}^N \bigl\|(W_{N-m}-W^0)(B_m(\varphi))\bigr\|_p
\end{equation*}
tends to zero.

\pagebreak[3]
\noindent \textup{2.} If $\varphi\in L^\infty(X,\nu)$, then
\begin{equation*}
S_N(x)=\frac1N\sum_{m=1}^N \bigl|(W_{N-m}-W^0)(B_m(\varphi))(x)\bigr|
\end{equation*}
tends to zero almost everywhere.
\end{claim}

The second term in~\eqref{eq:C2N} is estimated by $S_N$
(in $L^p$-norm) or by $S_N(x)$ (pointwise in absolute value).
Hence it remains to prove this claim to complete the proof of
Proposition~\ref{Prop:Conv}.

\begin{proof}[Proof of Claim~\ref{Clm:Wk}]
1. Let $C$ be a constant such that $\norm{W_k-W^0}_p\le C$ for all~$k$.
Then
\begin{equation*}
S_N\le\frac1N\sum_{m=1}^N
\norm{W_{N-m}-W^0}_p\cdot\norm{B_m(\varphi)}_p\le
\frac CN\sum_{m=1}^N\norm{B_m(\varphi)}_p
\end{equation*}
the latter is the Ces\`aro average (multiplied by $C$)
of the sequence $\norm{B_n(\varphi)}_p$, which tends to zero.

2. As in Claim~\ref{Clm:Mnto0}, denote $\varphi_m=B_m(\varphi)$,
$\psi_r(x)=\max_{m\le r}\abs{\varphi_m(x)}$.
Let constants $C$ and $R$ be such that $\norm{W_k-W^0}_\infty<C$ for all $k$
and $\norm{\varphi_m}_\infty\le R$ for all $m$.
Define the following ``exceptional sets''
\begin{align*}
E^1&{}=\{x\mid W^0(\psi_r)(x)\xnto{\!r\to\infty\!}0\},\\
E^2_r&{}=\Bigl\{x\Bigm|
\frac1N\sum_{k=0}^{N-1}(W_k-W^0)(\psi_r)(x)\xnto{\!k\to\infty\!}0\Bigr\}.
\end{align*}
and let $E=E^1\cup\Bigl(\bigcup_r E^2_r\Bigr)$.

Fix any $x\in X\setminus E$ and take any $\eps>0$. Choose $r_0$ such that
$W^0(\psi_{r_0})<\eps$.
Then
\begin{multline*}
S_N(x)=\frac1N\biggl(\sum_{m=1}^{r_0-1}+\sum_{m=r_0}^N\biggr)
\abs{(W_{N-m}-W^0)(\varphi_m)(x)}\le{}\displaybreak[1]\\
{}\le \frac{CR(r_0-1)}{N}+\frac1N\sum_{m=r_0}^N
\abs{W_{N-m}(\varphi_m)(x)}+\abs{W^0(\varphi_m)(x)}\le{}\displaybreak[2]\\
{}\le \frac{CR(r_0-1)}{N}+\frac1N\sum_{m=r_0}^N
\bigl(W_{N-m}(\psi_{r_0})(x)+W^0(\psi_{r_0})(x)\bigr)\le{}\displaybreak[2]\\
{}\le \frac{CR(r_0-1)}{N}+\frac1N\sum_{m=1}^N
\bigl(W_{N-m}(\psi_{r_0})(x)+W^0(\psi_{r_0})(x)\bigr)\le{}\\
\le\frac{CR(r_0-1)}{N}+2W^0(\psi_{r_0})(x)+\frac1N\sum_{k=0}^{N-1}(W_k-W^0)(\psi_{r_0})(x).
\end{multline*}
Here the first term tends to zero as $N\to\infty$, the second one is less
than $2\eps$,
and the last one also tends to zero (since $x\notin E^2_{r_0}$).
Hence for sufficiently large $N$ one has $S_N(x)\le 3\eps$.
\end{proof}

Therefore Proposition~\ref{Prop:Conv} is completely proven.
This completes the proofs of Lemma~\ref{Lem:OperPreConv},
Theorem~\ref{Thm:UVU*Avg}, and Theorem~\ref{Thm:MarkSemi}.
\end{proof}

\end{document}